\long\global\def\C#1\F{{}}
\documentclass[11pt]{article}
\usepackage{amsmath, amsthm, amssymb, graphicx, verbatim}
\usepackage[round]{natbib}
\swapnumbers
\input gen_macros_papers.dat
\begin{document}
\pagestyle{plain}

\title{\vspace*{-.5cm}\Huge Mediatic Graphs\thanks{We are grateful to David Eppstein for many useful exchanges pertaining to the results presented here.}}
\author{Jean-Claude Falmagne\thanks{Corresponding author: Dept. of Cognitive Sciences, University of California, Irvine, CA92697.} \hspace{4cm} Sergei Ovchinnikov\\
University of California, Irvine\hspace{2.3cm} San Francisco State University\\
\normalsize jcf@uci.edu\hspace{6cm}\normalsize sergei@sfsu.edu
}

\date{}
\maketitle

\thispagestyle{empty}

\vspace*{-.5cm}
\begin{abstract}
\fp
A medium is a type of semigroup on a set of states, constrained by strong axioms. Any medium can be represented as an isometric subgraph of the hypercube, with each token of the medium represented by a particular equivalence class of arcs of the subgraph. Such a representation, although useful, is not especially revealing of the structure of a particular medium. We propose an axiomatic definition of the concept of  a `mediatic graph'. We prove that the graph of any medium is a mediatic graph. We also show that, for any non-necessarily finite set $\SSS$, 
there exists a bijection from the collection $\mfM$ of all the media on $\SSS$ of states onto the collection $\mfG$ of all the mediatic graphs on 
$\SSS$. A change of framework for media is noteworthy: the concept of a medium is specified here in terms of two axioms, rather than the original four. 
\end{abstract}

\newpage
\section*{Background and Introduction}
The core concept of this paper can occur in the guise of various representations. Four of them are relevant here, the last one being new. 
\begin{roster}
\item A {\sc medium}, that is, a semigroup of transformations on a set of states, constrained by strong axioms \citep[see][]{falma97a, falma02}.
\item An {\sc isometric subgraph of the hypercube, or ``partial cube.''} By ``isometric'', we mean that the distance between any two vertices of the subgraph is identical to the distance between the same two vertices in the hypercube \citep[][]{graha71,djoko73}. Each  state of the medium is mapped to a vertex of the graph, and each transformation corresponds to an equivalence class of its arcs. Note that, as will become clear later on, no assumption of finiteness is made in this or in any of the other representation. 
\item An {\sc isometric subgraph of the integer lattice}.  This  representation is not exactly interchangeable with the preceding one. While it is true that any isometric subgraph of the hypercube is representable as an isometric subgraph of the integer lattice and vice versa, the latter representation lands in a space equipped with a considerable amount of structure. Notions of `lines', `hyperplanes', or `parallelism' can be legitimately defined if one wishes. Moreover, the dimension of the lattice representation is typically much smaller than that of the partial cube representing the same medium and so can be valuable in the representation of large media \citep[see, in particular,][in which an algorithm is described for finding the minimum dimension of a lattice representation of a partial cube]{eppst05}. 
\item A {\sc mediatic graph}.  Axiomatic definitions are usually regarded as preferable whenever feasible, and that is what is given here. 
\end{roster}

The definition of a medium is recalled in the next section, together with some key concepts and the consequences of the axioms that are useful for this paper. Note that two axioms are used, which are equivalent to the original four used by \citet{falma97a} \citep[see also][]{falma02, eppst02}. The graph of a medium and those graphs that induce media, called `mediatic graphs' are defined  and studied in the following two sections. The last two sections of the paper are devoted to specifying the correspondence between mediatic graphs and media, for a given possibly infinity set---of vertices or states depending on the case.
\vtl 

The subject of this paper may at first seem to be singularly ill chosen for a volume honoring Peter Fishburn's, as its topic does not readily evoke any of Peter's favorite concepts. But the enormously rich span of his accomplishment is not so easily escaped: indeed, the set of all interval orders 
\citep[][]{fishb71}
on any finite set is representable as a mediatic graph, and so is the set of all semiorders 
\citep[][]{fishb85, fishb99}
on the same set, these three citations heading  a list far too long to be included here\footnote{In view of the constraints set by the editors of this volume on the length of the many contributing papers.}. For the representability of families of interval orders or semiorders by mediatic graphs, see the concluding paragraph of this paper

\section*{The Concept of a Medium}
We begin with the terminology of `token systems' which provides a convenient framework.  
\begin{definition}\label{token system}  Let $\SSS$ be a set of {\sl states}. A {\sl token} is a function $\tau:S\mapsto S\tau$ mapping $\SSS$ into
itself. We shall use the abbreviations $S\tau=\tau(S),$ and
$
S\tau_1\tau_2\cdots\tau_n=\tau_n[\cdots\tau_2[\tau_1(S)]\cdots ]
$
for the function
composition. By definition, the identity function $\tau_0$  on $\SSS$
 is not a token. Let $\TTT $ be a set of tokens on~$\SSS$. 
The pair $(\SSS,\TTT )$ is called a {\sl token system}. We suppose that  $|\SSS|\geq 2$ and $\TTT\neq \es$.

Let $V$ and $S$ be two distinct states. Then $V$ is {\sl adjacent} to $S$ if $S\tau = V$ for some token $\tau$.\linebreak A token $\tilde\tau$ is a {\sl
reverse} of a token $\tau$ if, for any two adjacent states
$S$ and $V$, we have 
\begin{equation}\label{def reverse}
S\tau=V \quad\EQ\quad  V\tilde\tau=S,
\end{equation}
and thus $S\tau\tilde\tau = S$. 
It is straightforward that a token has at most one reverse.  If the reverse $\tilde\tau$ of a token $\tau$ exists, then $\tilde{\tilde\tau} = \tau$; that is, $\tau$ and $\tilde\tau$ are mutual reverses. 
If every
token has a reverse,  then  adjacency is a symmetric relation on
$\SSS$.
\end{definition}
\begin{definition}
A {\em message} is a string of elements of the set of tokens $\TTT$. The message $\tau_1\ldots\tau_n$ defines a function
$
S\mapsto S\tau_1\cdots\tau_n
$
on the set of states $\SSS$. If $\bm=\tau_1\ldots\tau_n$ denotes a message, we also (by abuse of notation) write $\bm=\tau_1\cdots\tau_n$ for the corresponding function.
No ambiguity will arise from this double usage.
\vtl
A message may consist in (the symbol representing) a single token. The {\sl content} of a message
${\bm}=\tau_1\ldots\tau_n$ is the set 
$\CCC({\bm})=\{\tau_1,\ldots,\tau_n\}$
of its tokens. We write 
$\ell({\bm})=n$ to denote the {\sl length} of the message $\bm$. (We have thus
$|\CCC(\bm)|\leq \ell(\bm)$.) A~message $\bm$ is {\sl effective} (resp.~{\sl
ineffective})  for a state $S$ if
$S{\bm}\neq S$ (resp.
$S{\bm}{\,=\, }S$) for the function\linebreak $S\mapsto S\bm$. A~message ${\bm}=\tau_1\ldots\tau_n$ is {\sl stepwise
effective} for
$S$ if  $S\tau_1\cdots\tau_k\neq S\tau_0\cdots\tau_{k-1}$, $1\leq k\leq n. $ A message which is both stepwise effective and ineffective for some state is called a {\sl return message} or, more brief, a {\sl return} (for that state).
\vtl
A message is {\sl consistent} if it does not contain both a token and its
reverse, and {\sl inconsistent} otherwise.  Two messages $\bm$ and $\bn$ are 
{\sl jointly consistent} if 
$\bm\bn$ (or, equivalently, $\bn\bm$) is consistent. A~consistent message which is 
stepwise effective for some state $S$ and does not have any of its token occurring more than once is said to be {\sl concise} (for~$S$).
A
message ${\bm}=\tau_1\ldots\tau_n$ is {\sl vacuous} 
 if the set of indices
$\{1,\ldots,n\}$ can be partitioned into pairs $\{i,j\},$ such that 
$\tau_i$ and $\tau_j$ are mutual reverses.
 By abuse of language, we sometimes call `empty' a place holder symbol that can be deleted, as in:
  `let $\bm\bn$ be a message in which $\bn$ is either a concise message or is empty' (that is $\bm\bn=\bm$).  If $\bm=\tau_1\ldots\tau_n$ is
a stepwise effective message  producing a state $V$ from a state $S$, then the {\sl reverse} of $\bm$ is defined by $\widetilde {\bm} =\tilde\tau_n\ldots\tilde\tau_1$.
We then have clearly $V\widetilde \bm = S$ and moreover $\tau \in \CCC(\bm)$ if and only if $\tilde\tau\in\CCC(\widetilde {\bm})$.  
\end{definition}

\begin{mediaaxioms} {\rm A token system $(\SSS,\TTT)$ is called a {\sl medium (on $\SSS$)} if the two following axioms  are satisfied.
\begin{roster}
\item[]
\begin{roster}
 \item[{[Ma]}] For any two distinct states $S,V$ in $\SSS$, there is a
concise message producing $V$ from $S$.
\item[{[Mb]}] Any return message  is vacuous.
\end{roster}
\end{roster}

A medium $(\SSS,\TTT)$ is {\sl finite} if $\SSS$ is a finite set. The concept of a medium was proposed by \citet{falma97a} who proved various basic facts about media. Other results were obtained by \citet{falma02} (see also Ovchinnikov and Dukhovny, 2000; Eppstein and Falmagne, 2002; Ovchinnikov, 2006). 
\tl

Four different axioms\footnote{\citet{falma97a} used  a slightly different definition of `reverse', allowing the possibility of several reverses for a given token. This was compensated by a stronger version of [M1] requiring the existence of a unique reverse for every token.}
 were used  in the papers cited above to define the concept of a medium, which are equivalent to 
Axioms [Ma] and [Mb]. Specifically, we have the following result:
}
\end{mediaaxioms}

\begin{theorem}\label{equiv axioms} A token system $(\SSS,\TTT)$ is a medium if and only if the following four conditions hold:
\begin{roster}
\item[]
\begin{roster}
 \item[{\rm [M1]}] Any token has a reverse.
\item[{\rm [M2]}] For any two distinct states $S,V$ in $\SSS$ there is a
consistent message transforming $S$ into $V$.
\item[{\rm [M3]}] A message which is stepwise effective for some state
is ineffective for that state if and only if it is vacuous.
\item[{\rm [M4]}] Two stepwise effective, consistent messages producing the same state are jointly consistent.
\end{roster}
\end{roster}
\end{theorem}

 We omit the simple proof of the equivalence between [Ma]-[Mb] and [M1]-[M4].

\section*{Some Basic Results}

 The material in this section, only part of which is new, is instrumental for the graph-theoretical results presented  in this paper. We omit the proofs of previously published results \citep[see][]{falma97a, falma02}. 

\begin{lemma}\label{prep lem} {\sl 
{\rm (i)} No token can be identical to its own reverse.
\vtl

{\rm (ii)} Let $\bm$ be a message that is concise for some state; we have then $l(m) = |C(m)|$ and\linebreak  $\CCC(\bm)\cap\CCC(\widetilde\bm) = \es$.
\vtl

{\rm (iii)}  For any two adjacent (thus, distinct) states $S$ and $V$, there is
exactly one token producing $V$ from $S$.
\vtl

{\rm  (iv)} No token
can be a 1-1 function.
\vtl

{\rm (v)} Suppose that $\bm$ and $\bn$ are stepwise
effective for $S$ and $V$, respectively, with $S\bm = V$ and $V\bn = W$. Then
$\bm\bn$ is stepwise effective for $S$, with $S\bm\bn = W$. 
\vtl

{\rm (vi)} Let $\bm$ and $\bn$ be two distinct concise messages transforming some state $S$. Then 
$$
S\bm =S\bn\quad\EQ\quad \CCC(\bm) = \CCC(\bn).
$$
}
\end{lemma}

Lemma \ref{prep lem}(vi) suggests an important concept.
\begin{definition}\label{def content} {\rm Let $(\SSS,\TTT)$ be a medium. For any state $S$, define
the {\sl (token) content} of $S$ as the set $\widehat  S$ of all tokens each of which
is contained in at least one concise message producing $S$; formally:
$$
\widehat S = \{\tau \in \TTT\st\exists V\in\SSS, V\bm =S, \text{ for } \bm \text{ concise with }
 \tau\in\CCC(\bm)\}.
$$ 
We refer to the family $\widehat \SSS$  of all the contents of the states in $\SSS$  as the {\sl content family} of the medium $(\SSS,\TTT)$.}
\end{definition}
\begin{remark} {\rm
Because any two stepwise effective, consistent messages producing the same
state must be jointly consistent (Condition [M4] in Theorem \ref{equiv axioms}), the content of a state cannot contain
both a token and its reverse.}
\end{remark}

Writing $\bigtriangleup$ for the symmetric set difference, and $+$ for the disjoint union, we have:
\begin{theorem}\label{V minus S = CV minus CS} {\sl If $S\bm =V$ for some concise message $\bm$ (thus $S\neq V$), then $\widehat  V\setminus \widehat  S = \CCC (\bm)$, and so $\widehat V\bigtriangleup \widehat S = \CCC(\bm) + \CCC(\widetilde\bm)$.
}
\end{theorem}
\begin{theorem}\label{|S| = |V|} {\sl For any token $\tau$ and any state
$S$, we have either $\tau \in \widehat S$ or $\tilde\tau \in \widehat S$;
so, $|\widehat S| =|\widehat V|$  for any two states $S$ and~$V$ with
$S = V$
if and only if
 $\widehat S =\widehat V$.
Moreover, if $\SSS$ is finite, then $|\widehat S| = |\TTT|/2$ for any $S\in\SSS$.
}
\end{theorem}
\begin{definition}\label{def orderly circuit} If $\bm$ and $\bn$ are two concise messages producing, from a state $S$, the same state $V\neq S$,  we call $\bm\widetilde\bn$ an {\sl orderly circuit} for $S$. 
\end{definition} 

By Axiom [Mb], an orderly circuit is vacuous; therefore its length must be even.
The following result  is of general interest for orderly circuits. 

\begin{theorem}\label{theorem theta} {\sl Let $S$, $N$, $Q$ and $W$ be four distinct states of a medium and suppose that
\begin{gather}\label{hypo lem theta}
N\tau = S,\quad W\mu = Q,\quad
S\bq = N\bq'=Q,\quad S\bw'=N\bw=W
\end{gather}
for some tokens $\tau$ and $\mu$ and some concise messages $\bq$, $\bq'$, $\bw$ and $\bw'$ (see Figure {\rm \ref{newproof}}). Then, the four following conditions are equivalent: 
\tl

{\rm (i)} $ \ell(\bq) + \ell(\bw) \neq  \ell(\bq') + \ell(\bw')$ and $\mu\neq\tilde \tau$;
\tl

{\rm (ii)}  $\tau = \mu$;
\tl

{\rm (iii)} $\CCC(\bq) = \CCC(\bw)$ and $\ell(\bq) = \ell(\bw)$;
\tl

{\rm (iv)}  $ \ell(\bq) + \ell(\bw) + 2 =  \ell(\bq') + \ell(\bw')$.
\tl

Moreover, any of these conditions implies that 
$\bq\tilde\mu\widetilde\bw\tau$ is an orderly circuit for $S$ with\linebreak  $S\bq\tilde\mu=S\tilde\tau\bw = W$. The converse does not hold.
}
\end{theorem}

{\begin{figure}[h]
\centerline{\includegraphics[scale=0.69]{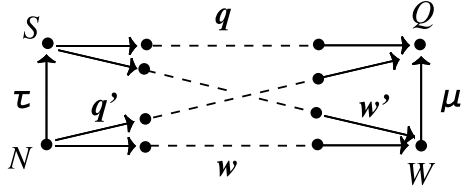}}
\caption[newproof]{\label{newproof} {\sc For Theorem \ref{theorem theta}. Illustration of the conditions listed in (\ref{hypo lem theta}).}}
\end{figure}
}

\begin{proof} 
We prove (i) $\Rightarrow$ (ii) $\eq$ (iii) $\Rightarrow$ (iv) $\Rightarrow$ (i).  
\tl

(i) $\Rightarrow$ (ii). Suppose that   $\tau\neq\mu$. The token $\tilde\tau$ must occur exactly once in either  $\bq$ or in $\widetilde\bw$. Indeed, we have $\mu\neq \tilde\tau$, both $\bq$ and $\bw$ are concise, and the message $\tau\bq\tilde\mu\widetilde\bw$ is  a return for $S$, and so is vacuous by [Ma]. It can be verified that each of the two mutually exclusive, exhaustive cases: 
[a] $\,\,\tilde\tau\in\CCC(\bq)\cap\CCC(\bw')$; and  {[b]} \,\,$\tilde\tau\in\CCC(\widetilde\bw)\cap\CCC(\widetilde\bq')$ lead to
\begin{equation}\label{4cases 0}
\ell(\bq) + \ell(\bw) = \ell(\bq') + \ell(\bw'),
\end{equation}
contradicting (i). Thus, we must have $\tau = \mu$. 
\tl

We only prove Case [a]. The other case is treated similarly.  Since $\tilde\tau$ is in $\CCC(\bq)$, neither $\tau$ nor 
$\tilde\tau$ can be in $\CCC(\bq')$. Indeed, both $\bq$ and $\bq'$ are concise and $\bq\widetilde\bq'\tau$ is a return for~$S$. It follows that both $\tilde\tau\bq'$ and 
$\bq$ are concise messages producing $Q$ from $S$. By Theorem~\ref{V minus S = CV minus CS}, we must have $\CCC(\tilde\tau\bq') = \CCC(\bq)$, which implies $\ell(\tilde\tau\bq') = \ell(\bq)$, and so 
\begin{equation}\label{4case 1}
 \ell(\bq) = \ell(\bq') + 1.
\end{equation}
A  argument along the same lines shows that
\begin{equation}\label{4case 2}
\ell(\bw) + 1 = \ell(\bw').
\end{equation}
Adding (\ref{4case 1}) and (\ref{4case 2}) and simplifying, we obtain (\ref{4cases 0}). The proof of Case [b] is similar.
\tl

(ii) $\eq$ (iii). If $\mu = \tau$, it readily follows (since both $\bq$ and $\bw$ are concise and $S\bq\tilde\tau\widetilde\bw\tau = S$) that any  token in $\bq$ must have a reverse in $\widetilde\bw$ and vice versa. 
This implies $\CCC(\bq) = \CCC(\bw)$, which in turn imply $\ell(\bq) = \ell(\bw)$, and so (iii) holds.  As $\bq\tilde\mu\widetilde\bw\tau$ is vacuous, it is clear that (iii) 
implies~(ii).
\tl

(iii) $\Rightarrow$ (iv). Since  (iii) implies (ii), we have $\tau\in \widehat Q\setminus\widehat N$ by Theorem \ref{V minus S = CV minus CS}. But both $\bq$ and $\bq'$ are concise, so $\tau\in \CCC(\bq')\setminus\CCC(\bq)$. As $\tau\bq\widetilde\bq'$ is vacuous for $N$,  we must have $\CCC(\bq) + \{\tau\}= \CCC(\bq')$, yielding 
\begin{equation}\label{bq = bq' + 1}
\ell(\bq) + 1= \ell(\bq').
\end{equation} 
A similar argument gives $\CCC(\bw) + \{\tau\} = \CCC(\bw') $ and 
\begin{equation}\label{bw = bw' + 1}
\ell(\bw) + 1= \ell(\bw').
\end{equation} 
Adding (\ref{bq = bq' + 1}) and  (\ref{bw = bw' + 1})  yields (iv).
\tl

(iv) $\Rightarrow$ (i). As (iv) is a special case of the first statement in (i), we only have to prove that $\mu\neq\tilde\tau$. Suppose that $\mu=\tilde\tau$.
We must assign the token $\tilde\tau$ consistently so to ensure the vacuousness of the messages $\bq\widetilde\bq'\tau$ and $\tau\bw'\widetilde\bw$. By Theorem~\ref{V minus S = CV minus CS}, $\CCC(\bq)=\widehat Q\setminus \widehat S$. Since $\tilde\tau\in\widehat Q$ and, by Theorem~\ref{|S| = |V|}, $\tilde\tau\notin\widehat S$, the only possibility is $\tilde\tau\in\CCC(\bq)\setminus\CCC(\bq')$. For similar reasons $\tau\in \CCC(\bw)\setminus\CCC(\bw')$. 
We obtain the two concise messages   $\tilde\tau\bq'$ and $\bq$ producing $Q$ from $S$, and the two concise messages  $\bw$ and $\tau\bw'$   producing $W$ from $N$. This gives $ \ell(\bq)=\ell(\tilde\tau\bq')$  and $\ell(\bw) = \ell(\tau\bw')$. We obtain so
$\ell(\bq) = \ell(\bq') + 1$ and
$\ell(\bw) =\ell(\bw') + 1$, which leads to
$
\ell(\bq) + \ell(\bw) = \ell(\bq') + \ell(\bw') + 2
$
and contradicts (iv). Thus,  (iv) implies (i). We conclude that the four conditions (i)-(iv) are equivalent.
\tl

We now show that, under the hypotheses of the theorem, (ii) implies that $\bq\tilde\mu\widetilde\bw\tau$ is an orderly return for $S$ with $S\bq\tilde\mu=S\tilde\tau\bw = W$. Both $\bq$ and $\bw$ are concise by hypothesis. We cannot have  $\mu$ in $\CCC(\bq)$ because then $\tilde\mu$ is in $\CCC(\widetilde\bq)$ and the two concise messages $\widetilde\bq$ and $\tau = \mu$ producing $S$ are not jointly consistent, yielding a contradiction of Condition [M4] in Theorem \ref{equiv axioms}. Similarly, we cannot have $\tilde\mu$ in $\CCC(\bq)$ since the two concise messages $\bq$ and $\mu$ producing $Q$ would not be jointly consistent. Thus, $\bq\tilde\mu$ is a concise message producing $W$ from $S$. For like reasons, with $\tau = \mu$, $\tilde\tau\bw$  is a concise message producing $W$ from $S$. We conclude that, with $\tau=\mu$, the message $\bq\tilde\mu\widetilde\bw\tau$ is an orderly return for $S$. The example of Figure \ref{newproof2}, in which we have
\begin{gather*}
\mu\neq\tau,\quad\bq=\aa\tilde\tau,\quad\bw=\tilde\mu\aa,\quad\bw'=\aa\tilde\tau\tilde\mu,\,\text{ and }\,\bq'=\aa,
\end{gather*}
displays the orderly return $\aa\tilde\tau\tilde\mu\tilde\aa\mu\tau$ for $S$. It serves as a counterexample to the implication: if  $\bq\tilde\mu\widetilde\bw\tau$ is an orderly return for $S$, then $\tau = \mu$.
\end{proof}

{\begin{figure}[h]
\centerline{\includegraphics*[scale=0.7]{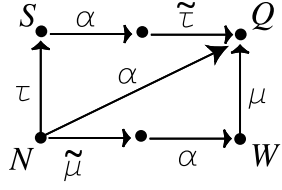}}
\caption[newproof2]{\label{newproof2} {\sc Under the hypotheses of Theorem \ref{theorem theta}, the hypothesis that $\bq\tilde\mu\widetilde\bw\tau$ is an orderly circuit for $S$ does not imply $\tau=\mu$, with $\bq = \aa\tilde\tau$, $\bw=\tilde\mu\aa$, $\bq'=\aa$, and $\bw'=\aa\tilde\tau\tilde\mu$.}}
\end{figure}
}

 In Definition \ref{def orderly circuit}, the concept of an orderly circuit was specified with respect to a particular state. The next definition and theorem concern a situation in which a circuit is orderly with respect to everyone of its states.
 In such a case, any token occurring in the circuit must have its reverse at the exact `opposite' place in the circuit (see Theorem \ref{theo opposite}(i)). 
\begin{definition}\label{regular}
Let $\tau_1\ldots\tau_{2n}$ be an orderly return for a state $S$.  For $1\leq i\leq n$, the two tokens $\tau_i$ and $\tau_{i+n}$ are called {\sl opposite}.
A return  $\tau_1\ldots\tau_{2n}$ from $S$ is {\sl regular}  if it is orderly and, for $1\leq i\leq n$, the message $\tau_i\tau_{i+1}\ldots\tau_{i+ n-1}$ is concise for $S\tau_1\cdots\tau_{i-1}$.
\end{definition}

\begin{theorem}\label{theo opposite} 
{\sl Let $\bm=\tau_1\ldots\tau_{2n}$ be an orderly return for some state $S$. Then the following three conditions are equivalent.
\tl

{\rm (i)} The opposite tokens of $\bm$ are mutual reverses.
\tl

{\rm (ii)} The return $\bm$ is regular.
\tl

{\rm (iii)} For $1\leq i\leq 2n-1$, the message $\tau_i\ldots\tau_{2n}\ldots\tau_{i-1}$ is an orderly return for the state $S\tau_1\cdots\tau_{i-1}$.
}
\end{theorem}
\def\imp{\Rightarrow}
{\sc Proof.} 
We prove  (i) $\imp$ (ii) $\imp$ (iii) $\imp$ (i). In what follows $S_i=S\tau_0\tau_1\ldots\tau_i$ for $0\le i\le 2n$, so $S_0=S_{2n}=S$.
\tl

(i) $\imp$ (ii). Since $\bm$ is an orderly return, for $1\le j\le n$, there is only one occurrence of the pair $\{\tau_j,\tilde\tau_j\}$ in $\bm$. Since $\tilde\tau_j=\tau_{j+n}$, there are no occurrences of $\{\tau_j,\tilde\tau_j\}$ in $\bp=\tau_i\cdots\tau_{i+n-1}$, so it is a  concise message for $S_{i-1}$.
\tl

(ii) $\imp$ (iii). Since $\bm$ is a regular return, any message $\bp=\tau_i\cdots\tau_{i+n-1}$ is concise, so any token of this message has a reverse in the message $\bq=\tau_{i+n}\ldots\tau_{2n}\ldots\tau_{i-1}$. Since $\bp$ is concise and $\ell(\bq)=n$, the message $\bq$ is concise. It follows that $\bp\bq$ is an orderly return for the state $S_{i-1}$.
\tl

(iii) $\imp$ (i). Since the message $\tau_i\ldots\tau_{2n}\ldots\tau_{i-1}$ is an orderly return for $S_{i-1}$, the messages $\bq=\tau_{i+1}\ldots\tau_{i+n-1}$ and $\bq'=\tau_i\ldots\tau_{i+n-1}$ are concise for the states $S'=S_i$ and $N=S_{i-1}$, respectively, and produce the state $Q=S_{i+n-1}$. Likewise, the messages $\bw=\tilde\tau_{i-1}\ldots\tilde\tau_{2n}\ldots\tilde\tau_{i+n}$ and $\bw'=\tilde\tau_i\ldots\tilde\tau_{2n}\ldots\tilde\tau_{i+n}$ are concise for the states $N=S_{i-1}$ and $S'=S_i$, respectively, and produce the state $W=S_{i+n}$. It is clear that $\ell(\bq)+\ell(\bw)+2=\ell(\bq')+\ell(\bw')$. By Theorem \ref{theorem theta}, $\tau_{i+n}=\tilde\tau_i$.
\EOP

\section*{The Graph of a Medium}
For graph-theoretical concepts and terminology, we usually follow \citet{bondy95}.
\begin{definition} \label{so1 media graph}
{\rm A {\sl graph representation} of a medium $(\SSS,\TTT)$ is a bijection $\gg:\SSS\rightarrow V$, where $V$ is a set of vertices of a graph $(V,E)$, such that two distinct states $S$ and $T$ are adjacent whenever $\{\gg(S),\gg(T)\}$ is an edge of the graph; formally,
\begin{equation}\label{def graph rep}
\hspace*{-.5cm} \{\gg(S),\gg(T)\}\in E\quad\EQ\quad(\exists\tau\in\TTT)(S\tau = T)\quad(S,T\in\SSS,\,S\neq T).
\end{equation}
We say then  that the  graph $(V,E)$, which has no loops, represents the medium. A graph $(V,E)$ representing a medium $(\SSS,\TTT)$ is called the {\sl graph of the medium} $(\SSS,\TTT)$  if  $V=\SSS$, the edges in $E$ are defined as in (\ref{def graph rep}), and  $\gg$ is the identity mapping. Clearly,  any medium has its  graph. We shall prove in this paper that the converse also holds, namely: the graph of a medium defines its medium 
(see Theorem \ref{1-1 cor M-G}). 
We recall that two graphs $(V,E)$ and $(V',E')$ are isomorphic if there is a bijection $\varphi:V\to V'$ such that
\begin{equation}\label{iso graph}
\{P,Q\} \in E\EQ\{\varphi(P),\varphi(Q) \}\in E'\qquad(P,Q\in E,\,P\neq Q).
\end{equation}
}
\end{definition}
\begin{lemma}\label{G iso G rep M}{\sl A graph isomorphic to a graph representing a medium $\MMM$ also represents~$\MMM$.}
\end{lemma}

It is intuitively clear that shortest paths in the graph of a medium correspond to concise messages of that medium. Our next lemma states that fact precisely.
\tl
\begin{lemma} \label{so1 shortest path}
{\sl Let $\gg:\SSS\rightarrow V$ be the representation of a medium
$(\SSS,\TTT)$ by a graph $G=(V,E)$.   If $\bm=\tau_1\ldots\tau_m$ is a concise message producing a state $T$ from a state $S$, then the sequence of vertices $(\gg (S_i))_{0\leq i\leq m}$, where $S_i=S\tau_0\tau_1\cdots\tau_i$, for $0\leq i\leq m$, forms a shortest path joining $\gg(S)$ and $\gg(T)$ in $G$. Conversely, if a sequence $(\gg(S_i))_{0\leq i\leq m}$ is a shortest path connecting $\gg(S_0)=\gg(S)$ and $\gg(S_m)=\gg(T)$, then $\bm=\tau_1\ldots\tau_m$ with
$S\tau_0\tau_1\cdots\tau_i = S_i$, for $0\leq i \leq m$, is a concise message producing $T$ from $S$.
}
\end{lemma}
\tl

{\sc Proof.} (Necessity.) Let $\gg(P_0)=\gg(S),\gg(P_1),\ldots,\gg(P_n)=\gg(T)$ be a path in $G$ joining $\gg(S)$ to $\gg(T)$. Correspondingly, there is a stepwise effective message $\bn=\rho_1\cdots\rho_n$ such that $P_i=T\rho_1\cdots\rho_{n-i}$ for $0\leq i<n$. 
The message $\bm\bn$ is a return for $S$. By Axiom [Mb], this message is vacuous. Since $\bm$ is a concise message for $S$, we must have
$
\ell (\bm) = m\leq \ell(\bn) = n.
$ 
\tl

(Sufficiency.)
Let $\gg(S_0)=\gg(S),\gg(S_1),\ldots,\gg(S_m)=\gg(T)$ be a shortest path from $\gg(S)$ to $\gg(T)$ in $G$. Then, there are some  tokens $\tau_i$, $1\leq i\leq m$ such that
$S_i\tau_{i+1}=S_{i+1}$ for $0\leq i< m$. The message $\bm=\tau_1\ldots\tau_m$ produces the state $T$ from the state~$S$. An argument akin to that used  in the foregoing paragraph shows that $\bm$ is a concise message for $S$. \EOP
\wl
We now establish a result of the same vein for the regular returns of a medium (cf.~Definition \ref{regular}). 
\begin{definition}\label{def circuit even minimal}{\rm  We recall that a sequence  of vertices
$\bs_m= (v_i)_{0\leq i\leq m}$ such that $\{v_i,v_{i+1}\}$ are edges
in a graph is a circuit if  $v_m = v_0$ and all the vertices $v_1,\ldots,v_m$ are different. By abuse of language, we say that the edges 
 $\{v_i,v_{i+1}\}$, for $0\leq i \leq m-1$, 
{\sl belong} to the circuit $\bs_m$. The circuit $\bs_m$ is {\sl even} if it has an even number of edges: $m = 2n$; any two of its edges $\{v_i,v_{i+1}\}$ and $\{v_{i+n},v_{i+n+1}\}$, $0\leq i\leq n-1$ are then called {\sl opposite}. A circuit  is {\sl minimal} if at least one shortest path between any two of its vertices is a segment of the circuit. A graph is {\sl even} if all its circuits are even.}
\end{definition}

\begin{lemma}\label{regular<=> min even}{\sl  Let $\gg:\SSS\rightarrow V$ be the representation of a medium $\MMM= (\SSS,\TTT)$ by a graph $G=(V,E)$. If $\bm=\tau_1\ldots\tau_{2n}$ is a regular return for some state $S\in \SSS$, then the sequence of vertices $(\gg (S_i))_{0\leq i\leq 2n}$, where $S_i=S\tau_0\tau_1\cdots\tau_i$, for $0\leq i\leq 2n$, forms an even, minimal circuit of~$G$ (with $S = S_0=S_{2n}$). Conversely, if a sequence $(\gg(S_i))_{0\leq i\leq 2n}$ is an even minimal circuit of~$G$, then  $\bm=\tau_1\ldots\tau_m$ with 
$S\tau_0\tau_1\cdots\tau_i = S_i$, for $0\leq i \leq 2n$ is a regular circuit for $S$ in $\MMM$. 
}
\end{lemma}
\tl

\begin{proof} In the notation of the lemma, let $\bm$ be a regular return for state $S$. Thus, by definition of a regular return (cf.~\ref{regular}), $\tau_1\ldots\tau_n$ and 
$\tilde\tau_{2n}\ldots\tilde \tau_{n+1}$
 are concise messages for $S$. By Lemma \ref{so1 shortest path}, the sequence of vertices $(\gg (S_i))_{0\leq i\leq n}$, where $S_i=S\tau_0\tau_1\cdots\tau_i$, for $0\leq i\leq n$, forms a shortest path joining $\gg(S)$ and $\gg(T)$, with $T = S\tau_1\cdots\tau_n$. Similarly, the sequence $\gg(S_{2n}),\gg(S_{2n-1}),\ldots,\gg(S_{n+1})$ is another shortest path joining $\gg(S)$ and $\gg(T)$. Since $\gg$ is a 1-1 function, all the vertices $\gg(S_i)$ are distinct, and so the sequence 
$(\gg (S_i))_{0\leq i\leq 2n}$ is an even circuit. This circuit is a minimal one. Indeed, by definition of a regular return, all the messages $\tau_i\tau_{i+1}\ldots\tau_{i+ n-1}$ are concise for $S\tau_1\cdots\tau_{i-1}$. So, by Lemma \ref{so1 shortest path}, all the sequences $\gg(S_i),\ldots,\gg(S_{i+n-1})$ are shortest paths between $\gg(S_i)$ and $\gg(S_{i+n-1})$, which implies that at least one shortest path between any two vertices of the circuit $(\gg (S_i))_{0\leq i\leq 2n}$ is a segment of that circuit. We omit the proof of the converse part of this lemma. The argument is based on the converse part of Lemma \ref{so1 shortest path} and is similar. 
\end{proof}

\begin{remark}\label{opposite--opposite}{\rm A close reading of this proof shows that opposite tokens $\tau_i$, $\tau_{i+n} =
\tilde\tau_i$ in a regular return correspond to opposite edges $\{\gg(S_i),\gg(S_{i+1})\}$, $\{\gg(S_{i+n},\gg(S_{i+1+n})\}$ in the even minimal circuit of the representing graph, with $S_{i+1}=S_i\tau_i$ and $S_{i+n}=S_{i+n+1}\tau_{i+n}$.}
\end{remark}

\section*{Media Inducing Graphs}

Our next task is to characterize the graphs representing media in terms of graph concepts. 
Some necessary conditions  are easily inferred from the axioms of a medium. For example, Axiom [Ma] forces the graph to be connected, and [Mb] demands that it is even. By convention, the graph should not have any loops. However, as shown by the two example below, these three conditions are not sufficient to characterize the graph of a medium.

\begin{twocounterexamples}\label{counter G1-3/4}{\rm The graphs corresponding to the digraphs A and B in Figure \ref{for[G3]} are connected and all their circuits are even. Moreover, they have no loops. Yet, neither A nor  B can yield the graph of a medium.  We leave to the reader to prove this for Figure \ref{for[G3]}A.
\tl

Here is why in the case of B. The circuit pictured in thick lines is even and minimal. By Lemma \ref{regular<=> min even}, it must represent a regular return in a medium. From Remark \ref{opposite--opposite}, we know that the same token must be matched to opposite edges of the circuit. Accordingly, the same token $\nu$ has been assigned to the arcs $JM$ and $RW$. (To simplify the figure, only one token from each pair of mutually reverse tokens is indicated.) The circuit  containing the six vertices $L,K,N,W,R$ and $H$ is also even and minimal. Thus, the arcs $LK$ and $RW$ must be assigned the same token, and since $RW$ has been assigned token $\nu$, that token must also be assigned to $LK$. The argument governing the placement of the token $\tau$ are similar. The consequence, however, is that there is no concise message from $L$ to $J$: any message producing $J$ from $L$ contains either both $\nu$ and $\tilde\nu$, or both $\tilde\tau$ and $\tau$. This example will be crucial in our understanding of the appropriate axiomatization of a graph capable of representing a medium.

{\begin{figure}[h]
\centerline{\includegraphics*[scale=0.5]{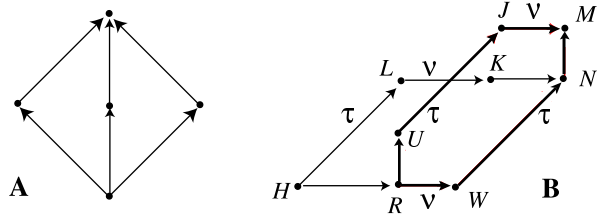}}
\caption[for[G3]{\label{for[G3]} {\sc Neither of these graphs is that of a medium. The   token system corresponding to Digraph $\bold B$ contradicts {\rm [Ma]}. Which of the properties of a medium is contradicted by Digraph $\bold A$?}}
\end{figure}
}
}
\end{twocounterexamples}

In our failed attempt at representing a medium in Figure \ref{for[G3]}, we have chosen to picture the arcs  representing the same token by parallel segments (forming two sides of an implicit rectangle). The intuition that the opposite arcs of even minimal circuits should be parallel is a sound one, and suggests the construction of an equivalence relation on the set of set of arcs of the digraph. Such a construction is delicate, however, and the two examples of media pictured below by their digraphs must  be taken into  account.

\begin{examples}{\rm Together with the examples of Figure \ref{for[G3]}, Examples A and B in Figure \ref{remarks-para} will also guide and illustrate our choice of concepts and axioms. 
{\begin{figure}[h]
\centerline{\includegraphics*[scale=0.42]{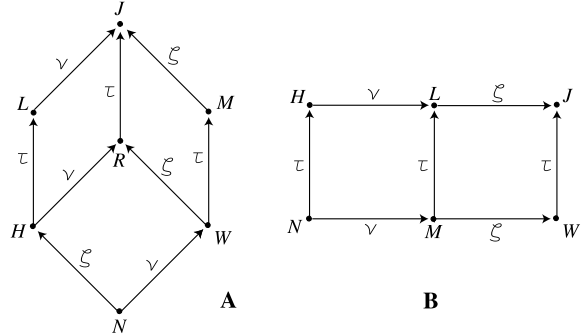}}
\caption[remarks-para]{\label{remarks-para} {\sc Two examples of graphs of media.
In Example B, notice that different tokens are assigned to the opposite arcs $HL$ and $MW$. This circuit is not minimal. Compare with the situation of the arcs $LJ$ and $NW$ in Example~A.  }}
\end{figure}
}
}
\end{examples}

\begin{definition}\label{def mho}{\rm  We write $\vec{E} = \{ST\st \{S,T\}\in E\}$ for the set of all the arcs of a graph $G=(V,E)$. The {\sl like} relation of the graph $G$ is a relation $\bmfL$ on $\vec E$ defined by
\begin{equation*}
ST\bmfL PQ\,\EQ\,(\dd(S,P)+1 =\dd(T,Q)+ 1= \dd(S,Q) = \dd(T,P))\quad\qquad(\{S,T\},\,\{P,Q\}\in E),
\end{equation*}
where $\dd$ denotes the graph theoretical distance between the vertices of the graph. In 
Example B of Figure~\ref{remarks-para},
 we have $NH\bmfL WJ$ because
$$
\dd(H,J)+1 =\dd(N,W)+ 1= \dd(H,W) = \dd(N,J),
$$ 
but $HL\bmfL MW$ does not hold since
$$
\dd(H,M) = \dd(L,W) = 2\quad\text{and}\quad 4=\dd(H,W) \neq \dd(L,M) = 1.
$$
The point is that the arcs $HL$ and $MW$ are opposite in the circuit  $H,L,J,W,M,N,H$, but this circuit is not minimal.
\tl

The like relation is clearly reflexive and symmetric; and moreover
\begin{equation}\label{HN = PQ <=> NH = QP}
ST\bmfL PQ\quad\EQ\quad TS\bmfL QP\quad\qquad(\{S,T\},\,\{P,Q\}\in E).
\end{equation}
Two binary relations on the set of edges of a graph play a central role in characterizing partial cubes. They are Djokovi\'{c}'s relation~$\theta$~\citep{djoko73} and Winkler's relation $\Theta$~\citep{winkl84}. These relations are germane to, but different from the like relation of this paper. Space limitation prevents us to specify the relationship here.
}
\end{definition}
We now come to the main concept of this paper. We recall that a graph is bipartite if and only if it is even \citep[][]{konig16}.

\begin{definition}\label{def mediatic}{\rm
Let $G=(V,E)$ be a graph equipped with its like relation~$\bmfL$. The graph $G$ is called {\sl mediatic} if the following three axioms hold.
\begin{roster}
\item[{[G1]}] $G$ is connected.
\item[{[G2]}] $G$ is bipartite.
\item[{[G3]}] $\bmfL$ is transitive.
\end{roster}
The set of vertices is not assumed to be finite. It is easily verified that any graph isomorphic to a mediatic graph is mediatic. 
}
\end{definition}

Axiom [G3] eliminates the counterexample of Figure \ref{for[G3]}B.  Indeed, since
\begin{gather*}
\dd(L,J) = 4,\quad \dd(K,M) = 2,\quad \dd(L,M) = 3= \dd(K,J)\\
\noalign{we have}
LK\bmfL RW\bmfL JM\quad\text{but not}\quad 
LK\bmfL JM.
\end{gather*}

The following result is immediate.
\begin{lemma}\label{trect equiv rel}{\sl The like relation $\bmfL$ of a mediatic graph $(V,E)$ is an equivalence relation on~$\vec E$.}
\end{lemma}
\begin{definition}\label{rect equiv classes}{\rm We denote by 
$$
\langle ST\rangle =\{PQ\in\vec{E}\st ST\bmfL PQ\}
$$ 
 the equivalence class containing the arc $ST$ in the partition of $\vec E$ induced by $\bmfL$.}
\end{definition} 

We will show that a graph representing a medium is mediatic (see Theorem  \ref{media graph => mediatic}). Our next lemma is the first step. 

\begin{lemma}\label{lem g mediatic}{\sl Let $\gg$ be the representation of a medium $\MMM = (\SSS,\TTT)$  by a graph $G=(\SSS,E)$ which is equipped with its like relation $\bmfL$. 
Suppose that $\gg(N)\gg(S)\bmfL \gg(W)\gg(Q)$. Then $N\tau= S$ and $W\tau = Q$ for some $\tau\in\TTT$. In fact, there exists an orderly circuit $\bq\tilde\tau\widetilde\bw\tau$ for $S$ in $\MMM$, with $S\bq\tilde\tau=S\tilde\tau\bw=W$; thus $\bq$ and $\bw$ are concise with $\ell (\bq) = \ell(\bm)$. Such a circuit is not necessarily regular.
}
\end{lemma}

{\sc Proof.} We abbreviate our notation for this proof, and write $S^{\gg} = \gg(S)$ for all $S\in\SSS$.
By definition, $N^{\gg}S^{\gg}\bmfL W^{\gg}Q^{\gg}$ implies that $\dd(S^{\gg},Q^{\gg}) = \dd(N^{\gg},W^{\gg}) = \dd(N^{\gg},Q^{\gg}) -1 = \dd(S^{\gg},W^{\gg}) -1$; so, there are, for some $n\in\Na$, two shortest paths
$$
S^{\gg}_0=S^{\gg}, S^{\gg}_1,\ldots,S^{\gg}_n = Q^{\gg}\quad\text{and}\quad N^{\gg}_0 = N^{\gg}, N^{\gg}_1,\ldots, N^{\gg}_n = W^{\gg}
$$
between $S^{\gg}$ and $Q^{\gg}$, and $N^{\gg}$ and $W^{\gg}$, respectively. Moreover, 
$$
S^{\gg}_0=S^{\gg}\hspace{-1pt}, S^{\gg}_1\hspace{-1pt},\ldots\hspace{-1pt},S^{\gg}_n = Q^{\gg}\hspace{-1pt},W^{\gg}\,\text{ and }\, N^{\gg}_0 = N^{\gg}\hspace{-1pt}, N^{\gg}_1\hspace{-1pt},\ldots\hspace{-1pt}, N^{\gg}_n = W^{\gg}\hspace{-1pt},Q^{\gg}
$$
are also shortest paths.
Using Lemma 
\ref{so1 shortest path}, we can assert the existence of two concise messages $\bq$ and $\bw$ such that $S\bq= Q$ and $N\bw= W$, with $\ell (\bq) = \ell(\bw) = n$. Also, for some tokens $\tau$ and $\mu$, we have $N\tau = S$ and $W\mu = Q$ with $\bq'=\tau\bq$ and $\bw'=\tilde\tau\bw$ concise for $N$ and $S$, respectively, and
$\ell (\bq') = \ell(\bw') = n+1$. We are exactly in the situation of Theorem \ref{theorem theta} (see Figure \ref {newproof}). Using the implication (iv) $\Rightarrow$ (ii) of this theorem,  we obtain $\tau = \mu$. Condition (iv) also implies that $\bq\tilde\tau\widetilde\bw\tau$ is an orderly circuit for $S$, with $S\bq\tilde\tau=S\tilde\tau\bw=W$. The example of Figure \ref{not_regular} shows that, with $\bq =\bw= \nu\zeta$,  such a circuit need not be regular.
\EOP

{\begin{figure}[h]
\centerline{\includegraphics*[scale=0.53]{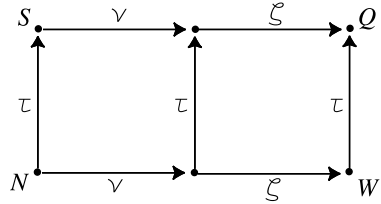}}
\caption[not_regular]{\label{not_regular} {\sc Under the hypotheses of Lemma \ref{lem g mediatic}, 
with $NS\bmfL WQ$, the orderly circuit $\bq\tilde\tau\tilde\bw\tau= \nu\zeta\tilde\tau\tilde\nu\tilde\zeta\tau$ for $S$ is not regular. For example, $\zeta\tilde\tau\tilde\zeta$ is not concise for $S\nu$ (cf.~Definition \ref{regular}).}}
\end{figure}
}

\begin{convention}\label{med g equipped}{\rm Any graph representing a medium comes implicitly equipped with its like relation $\bmfL$. When several such graphs are considered (say, for different media), their respective like relations are distinguished by diacritics, such as $\bmfL'$ or~$\bmfL^*$.}
\end{convention} 

\begin{theorem}\label{media graph => mediatic}{\sl Any graph  representing a medium  is mediatic.}
\end{theorem}

{\sc Proof.}   Because any graph isomorphic to a mediatic graph is mediatic, we can invoke Lemma 
\ref{G iso G rep M} and content ourselves with proving that  the graph of a medium  is mediatic (which simplifies our notation). Denote the medium by $\MMM=(\SSS,\TTT)$, and let $G=(\SSS,E)$ be its graph. We prove that $G$ satisfies [G1], [G2] and [G3].
\begin{roster}
\item[{[G1]}] Axiom [Ma] requires that $G$ be connected.
\item[{[G2]}]  Axiom [Mb] implies that $G$ must be even. Hence, by K\"{o}nig's Theorem, it must be bipartite. 
\item[{[G3]}] Suppose that $NS\bmfL PR\bmfL WQ$. By Lemma \ref{lem g mediatic} (applied twice), there must be some tokens $\tau$ and $\mu$ such that $N\tau = S$, $P\tau = R$, $P\mu = R$ and $W\mu = Q$, so $\tau = \mu$. Let then $\bq$ and $\bw'$ be two concise messages from $S$, and let $\bw$ and $bq'$ be two concise messages from $N$, such that
$$
S\bq = Q,\quad S\bw' = W, \quad N\bw = W,\quad N\bq' = Q.
$$
The situation is exactly as in Theorem \ref{theorem theta}, with the same notation. 
Because $\tau = \mu$, Condition (ii) of this theorem holds. We conclude that Conditions (iii) and (iv) also hold, which leads to 
$$
\dd(S,Q) + 1=\dd(N,W) + 1 = \dd(S,W) = \dd(N,Q).
$$
We have thus $NS\bmfL WQ$; so Axiom [G3] holds. \EOP
\end{roster}

We omit the proof of the next lemma, which is straightforward.

\begin{lemma}\label{isobmfL <=> isograph}{Let $G=(V,E)$ and $G'=(V',E')$ be two mediatic graphs, with their respective like relations $\bmfL$ and $\bmfL'$, and let $\varphi$ be a bijection of  $V$ onto $V'$. 
Then $\varphi$ is an isomorphism of $G$ onto $G'$ if and only if
$$
ST\bmfL PQ\quad\EQ\quad\varphi(S)\varphi(T)\bmfL'\varphi(P)\varphi(Q)\quad\quad(S,T,P,Q\in V).
$$
}
\end{lemma}
\begin{remark} {\rm The like relation is the fundamental tool for the study of mediatic graphs. We shall see that any mediatic graph $G$ can be used to construct a medium $\MMM$ that has $G$ as its graph. Each of the equivalence classes $\langle ST\rangle$ of the like relation contains `parallel' arcs of  the graph, and will turn out to correspond to a particular token, say $\tau$, of the medium under construction, with the class $\langle TS\rangle$ corresponding to the reverse token~$\tilde\tau$. Before proceeding to such a construction, we establish in Theorem \ref{iso media => iso graphs} a useful result which precisely links the isomorphism of media to that of their graphs. 
}
\end{remark}
\begin{definition} Two media $(\SSS,\TTT)$ and $(\SSS',\TTT')$ are isomorphic if there exists a pair $(\aa,\bb)$ of bijections $\aa:\SSS\to\SSS'$ and $\bb: \TTT\to \TTT'$ such that 
\begin{equation}\label{iso-media-eq}
S\tau = V\EQ \aa(S)\bb(\tau) = \aa(V)\quad\qquad(S,V\in \SSS,\,\tau\in\TTT).
\end{equation}
\end{definition}
\vspace*{-.5cm}
\section*{Paired Isomorphisms of Media and Graphs}

Isomorphic media  yield isomorphic mediatic graphs, and vice versa.

\begin{theorem}\label{iso media => iso graphs} {\sl Suppose that $\MMM=(\SSS,\TTT)$ and $\MMM'=(\SSS',\TTT')$ are two media and let $G=(\SSS,E)$ and $G'=(\SSS',E')$ be their respective graphs. Then $\MMM$ and $\MMM'$ are isomorphic if and only if $G$ and $G'$  are isomorphic; more precisely:
\tl

{\rm (i)} if $(\aa,\bb)$ is an isomorphism of $\MMM$ onto $\MMM'$, then $\aa:\SSS\to\SSS'$ is an isomorphism of $G$ onto $G'$ in the sense of {\rm (\ref{iso graph})};
\tl

{\rm (ii)} if $\varphi: \SSS\to\SSS'$ is an isomorphism of $G$ onto $G'$  in the sense of {\rm (\ref{iso graph})}, then there exists a bijection $\bb:\TTT\to\TTT'$ such that $(\varphi,\bb)$ is an isomorphism of $\MMM$ onto~$\MMM'$.
}
\end{theorem} 
\tl
{\sc Proof.} (i) Suppose that $(\aa,\bb)$ is an isomorphism of $\MMM$ onto~$\MMM'$. For any two distinct $S$, $T$ in $\SSS$, we have successively
\begin{align*}
\{S,&T\}\in E&&\\
&\EQ (\exists \tau\in\TTT)(S\tau = T)\quad&&(\text{$G$ is the graph of $\MMM$})\\
&\EQ (\exists \tau\in\TTT)(\aa(S)\bb(\tau) = \aa(T))&&(\text{$\MMM$ and $\MMM'$ are isomorphic})\\
&\EQ \{\aa(S),\aa(T)\}\in E'&&(\text{$G'$ is the graph of $\MMM'$}),
\end{align*}
and so
$$
\{S,T\}\in E\,\EQ\, \{\aa(S),\aa(T)\}\in E'\quad\qquad(S,T\in\SSS,\,S\neq T).
$$
We conclude that $\aa:\SSS\to\SSS'$ is an isomorphism of $G$ onto $G'$.
\tl

(ii) Let $\varphi: \SSS\to\SSS'$ be an isomorphism of $G$ onto $G'$. Define a function 
$\bb:\TTT\to\TTT'$ by  
\begin{equation}\label{def bb iso G => iso media}
\hspace*{-.4cm}\bb(\tau) = \tau'\quad\EQ\quad (\forall S,T\in\SSS)(S\tau = T\eq \varphi(S)\tau'=\varphi(T)).
\end{equation} 
We first verify that the r.h.s.~of the equivalence  (\ref{def bb iso G => iso media}) correctly defines $\bb$ as a bijection of $\TTT$ onto $\TTT'$. For any $\tau\in\TTT$, there exists distinct states $S$ and $T$ in $\SSS$ such that $S\tau = T$ and $\{S,T\}\in E$. Fix $S$ and $T$ temporarily. By the isomorphism $\varphi:\SSS\to\SSS'$ of $G$ onto $G'$, we have $\{\varphi(S),\varphi(T)\}\in E'$, and because $G'$ is the graph of $\MMM'$, we necessarily have $\varphi(S)\tau'=\varphi(T)$ for some $\tau'\in\TTT'$, which is unique by Lemma \ref{prep lem}(i). The hypothesis that $\varphi$ is an isomorphism of $G$ onto $G'$ ensures that the r.h.s.~of (\ref{def bb iso G => iso media}) is indeed an equivalence.  

Next, we show that $\bb(\tau)$ does not depend upon the choice of $S$ and $T$. 
Let $P,Q$ be another pair of distinct states in $\SSS$ such that $P\tau= Q$, and let $P=S\bm$ and $Q=T\bn$ for some concise messages $\bm=\tau_1\ldots\tau_m$ and $\bn=\mu_1\cdots\mu_n$. By Condition [M4], $\tau\bn$ and $\bm\tau$ are concise messages, and so Theorem~\ref{theorem theta} applies. Invoking its implication (ii) $\Rightarrow$ (iii), we get $\ell(\bm)=\ell(\bn)$ and $\CCC(\bm) = \CCC(\bn)$, yielding $m=n$. Denote by $\bmfL$ and $\bmfL'$ the like relations of $G$ and $G'$ respectively. We have thus shown that $ST\bmfL PQ$. By Lemma \ref{isobmfL <=> isograph}, we also have 
$$
\varphi(S)\varphi(T)\bmfL' \varphi(P)\varphi(Q).
$$
Since we have $\varphi(S)\tau' = \varphi(T)$, we can apply Lemma \ref{lem g mediatic} and derive 
$\varphi(P)\tau' = \varphi(Q)$.
\vtl

We still have  to prove that $\bb$ is indeed a bijection. For any $\tau'\in\TTT'$ there are some $S',T'\in\TTT'$ such that $S'\tau' = T'$. We have thus $\{S',T'\}\in E'$, and since $\varphi$ is an isomorphism of $G$ onto $G'$, also $\{\varphi^{-1}(S'),\varphi^{-1}(T')\}\in E$, with $\varphi^{-1}(S')\tau = \varphi^{-1}(T')$ for some $\tau \in\TTT$. Thus $\bb$ maps $\TTT$ onto $\TTT'$. Suppose now that $\bb(\tau) = \bb(\mu) = \tau'\in\TTT'$. This implies that for some $S,T,P,Q\in \SSS$ and $N,M\in\SSS'$,
we must have
\begin{gather}\label{bb is 1-1}
S\tau = T, \quad P\mu = Q,\text{ and }N\tau'=M,
\end{gather}
together with
$\varphi(S) = \varphi(P) = N$ and $\varphi(T) = \varphi(Q) = M$
by the definition of $\bb$. As $\varphi$ is a 1-1 function, we obtain $S=P$ and $T=Q$ in 
(\ref{bb is 1-1}).   Using Lemma \ref{prep lem}(ii), we get $\tau = \mu$. Thus, $\bb$ is a 1-1 function and so a bijection.

The fact that $(\varphi,\bb)$ is an isomorphism of $\MMM$ onto $\MMM'$ follows from the definition of $\bb$ by (\ref{def bb iso G => iso media}). We have
$$
S\tau = T\EQ \varphi(S)\bb(\tau)=\varphi(T)\quad\qquad(S,T\in\SSS)
$$
whether or not $\{S,T\}\in E$.\EOP
\wl

Having defined the graph of a medium and shown that such a graph was necessarily mediatic, we now go in the opposite direction and construct a medium from an arbitrary mediatic graph.

\section*{From Mediatic Graphs to Media}

\begin{definition}\label{graph -> token} {\rm Let $G=(\SSS,E)$ be a mediatic graph and let $\bmfL$  be its like  relation. For any $ST\in\vec{E}$, define a transformation 
$
\tau_{_{ST}}: \SSS\to\SSS\,:\,P\mapsto P\tau_{_{ST}}
$
by the formula
\begin{equation}\label{def tau}
P\tau_{_{ST}} = \begin{cases} Q\text{ if } ST\bmfL PQ,\\
P\text{ otherwise}.
\end{cases}
\end{equation}
We denote by $\TTT=\{\tau_{_{ST}}\st ST\in\vec{E}\}$ the set containing all those transformations. It is clear that the  pair $(\SSS,\TTT)$ is a token system. Such a token system is said to be {\sl induced} by the mediatic graph~$G$. The theorem below establishes that a token system $\KKK$ induced by a mediatic graph $G$ is in fact a medium. We say that $\KKK$ is the  {\sl  medium of the graph}~$G$.  Notice that, since $\bmfL$ is an equivalence relation on $\vec{E}$, we have $\tau_{_{ST}} = \tau_{_{PQ}}$ whenever $ST\bmfL PQ$. In such a case, we have in fact $\langle ST\rangle = \langle PQ\rangle$. The choice of a particular pair $ST\in\langle PQ\rangle$ to denote a token $\tau_{_{ST}}$ is thus arbitrary.  Notice that, as a consequence of this definition,  whenever $\{S,T\}\in E$, then also $ST\bmfL ST$, and so $S\tau_{_{ST}} = T$.
}
\end{definition}
This construction is motivated by the following theorem.

\begin{theorem}\label{pair induced by G}{\sl  The token system $(\SSS,\TTT)$ induced by a mediatic graph\ $G = (\SSS,E)$ is a medium. In particular,  the tokens $\tau_{_{ST}}$ and $\tau_{_{TS}}$ defined by {\rm (\ref{def tau})} are mutual reverses for any $\{S,T\}\in E$.}
\end{theorem}

{\sc Proof.} We verify that  $(\SSS,\TTT)$ satisfies Axioms [Ma] and [Mb] of a medium.

[Ma] For any  $S,T\in\SSS$, there is a shortest path $S_0=S, S_1,\dots,S_n =~T$ between $S$ and $T$ in $G$. This implies that, for $0\leq i\leq n-1$, we have $\{S_i,S_{i+1}\}\in E$, which yields  $S_i\tau_{_{S_iS_{i+1}}}=S_{i+1}$. It follows that the message $\bm=\tau_{_{S_0S_1}}\ldots\tau_{_{S_{n-1}S_n}}$ produces $T$ from~$S$ and is stepwise effective. To prove that $\bm$ is concise, we must still show that it is consistent and without repetitions. 
The message $\bm$ is consistent since otherwise we would have 
\begin{equation}
\label{proof M2}
S_h\tau_{_{MN}} = S_{h+1}\quad\text{and}\quad S_k\tau_{_{NM}} = S_{k+1}
\end{equation}
for some indices $h$ and $k$, with $h<k$,   and some $NM\in\vec E$. Since $\tau_{_{MN}}$ is the reverse of $\tau_{_{NM}}$, the last equality in (\ref{proof M2}) can be rewritten as $ S_{k+1}\tau_{_{MN}} = S_{k}$. Thus, by definition of the tokens in (\ref{def tau}), the above statement (\ref{proof M2}) leads to
$S_hS_{h+1}\bmfL MN\bmfL S_{k+1}S_k$ which, by transitivity, gives
$
S_kS_{k+1}\bmfL S_{h+1}S_h\,.
$ Because $h<k$, we derive by the definition of the like relation $\bmfL$ 
$$ 
k+1 - h = \dd(S_{k+1},S_h) = \dd(S_k, S_{h+1}) = k -1 -h
$$ 
yielding the absurdity $1 = -1$. Thus, $\bm$ is consistent. Suppose that $\bm$ has repeated tokens, say $S_i\tau_{_{S_iS_{i+1}}} = S_{i+1}$ and $ S_{i+k}\tau_{_{S_iS_{i+1}}} = S_{i+k+1}$
for some indices $0\leq i < n$ and $0\leq i+k< n$. This would give
 $S_iS_{i+1}\bmfL S_{i+k}S_{i+k+1}$, leading to 
$$
d(S_i,S_{i+k+1}) = k+1 > k-1 = d(S_{i+1},S_{i+k}),
$$
while by the definition of $\bmfL$ we should have
$d(S_i,S_{i+k+1}) = d(S_{i+1},S_{i+k})$, a contradiction. Thus, the message $\bm$ is concise.
\vspace{.2cm}

[Mb] Let $\bm=\tau_{_{S_0S_1}}\tau_{_{S_1S_2}}\ldots\tau_{_{S_{n-1}S_n}}$ be a return message for some state~$S$; we have thus $S_0=S_n = S$. In the terminology of $G$, we have a closed walk $S=S_0,S_1,\ldots,S_n=S$.  We denote this closed walk by $\BW$ and we write $\vec{E}_\BW$ for the set of all its arcs $S_iS_{i+1}$, $0\leq i \leq n-1$.
By [G2] and K\"{o}nig's Theorem, such a closed walk is even; so $n = 2q$ for some $q\in\Na$. We prove by induction on $q$ that $\bm$ is vacuous. The case $q=1$ (the smallest possible return) is trivial, so we suppose that [Mb] holds for any $1\leq p < q$ and prove that [Mb] also holds for $q= p$. We consider two cases.
	
Case 1: $\BW$ is an isometric subgraph of $G$.  Thus, $\BW$ is a minimal circuit of $G$. Take any token 
$\tau_{_{S_{i}S_{i+1}}}$ in $\bm$. Since (with the addition modulo $k$ in the indices), we have for
$0\leq i < n$
\begin{align*}
\dd(S_{i+1},S_{i+k}) &= \dd(S_{i},	S_{i+k+1}) = k-1,\\ \dd(S_{i},S_{i+k}) &= \dd(S_{i+1},S_{i+k+1}) = k,
\end{align*}
we obtain $S_iS_{i+1}\bmfL S_{i+k+1}S_{i+k}$.  By the definition of the tokens in (\ref{def tau}) and the transitivity and symmetry of $\bmfL$, we get  for any $P,Q\in~\SSS$
\begin{align*}
P\tau_{_{S_iS_{i+1}}} = Q&\EQ S_iS_{i+1}\bmfL PQ\\
&\EQ S_{i+k+1}S_{i+k}\bmfL PQ \\
&\EQ P\tau_{_{S_{i+k+1}S_{i+k}}} = Q\\
&\EQ Q\tau_{_{S_{i+k}S_{i+k+1}}} = P.
\end{align*}
We conclude that $\tau_{_{S_{i+k}S_{i+k+1}}}$ and $\tau_{_{S_iS_{i+1}}}$ are mutual reverses, and so $\bm$ is vacuous. (Note that the induction hypothesis has not been used here.)
\tl

Case 2: $\BW$ is not an isometric subgraph of $G$. Then, there are two vertices $S_i$ and $S_j$ in $\BW$, with $i<j$, 
and a shortest path  $\BL$ from $S_i$ to $S_j$ in $G$ with $\dd_{ij} = \dd(S_i,S_j) < \min\{j-i, i + n-j\}$ (see Figure \ref{AxiomM3}). Thus, $j-i$ and $i+n-j$ are the lengths of the two segments of $\BW$ with endpoints $S_i$ and $S_j$. For simplicity, we can assume without loss of generality that $S_i$ and $S_j$ are the only vertices of $\BL$ that are also in $\BW$. Let $\bp$ the straight message producing $S_j$ from $S_i$ and corresponding to the shortest path $\BL$ in the sense of Lemma \ref{so1 shortest path}. 

{\begin{figure}[h!]
\centerline{\includegraphics*[scale=0.78]{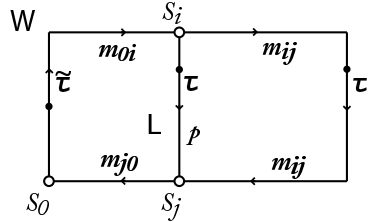}}
\caption[AxiomM3]{\label{AxiomM3} {\sc Case 2 in the proof of Axiom [Mb] in Theorem 
\ref{pair induced by G}: the closed walk $\BW$ is not an isometric subgraph.}}
\end{figure}
}

 We also split $\bm$ into the three messages:
\begin{align*}
\bm_{0i} & = \tau_{_{S_0S_1}}\ldots\tau_{_{S_{i-1}S_i}}\\
\bm_{ij}  & = \tau_{_{S_{i}S_{i+1}}}\ldots\tau_{_{S_{j-1}S_j}}\\
\bm_{j0} & = \tau_{_{S_{j}S_{j+1}}}\ldots\tau_{_{S_{n-1}S_0}}\,.
\end{align*}
We have thus $\bm = \bm_{0i}\bm_{ij}\bm_{j0}$. Note that the two messages $\bm_{0i}\bp\bm_{j0}$ and $\widetilde\bp\bm_{ij}$  have a length strictly smaller that $n= 2q$. By the induction hypothesis, these two messages are vacuous. Accordingly, for any token $\tau$ of $\bp$, there is an reverse token $\tilde\tau$ either in $\bm_{0i}$ or in $\bm_{j0}$. (In Figure \ref{AxiomM3} the token $\tilde\tau$ is pictured as being part of $\bm_{0i}$.) Considered from the viewpoint of the message  $\widetilde\bp\bm_{ij}$ from $S_j$, the token $\tilde\tau$ is in $\widetilde \bp$ with its reverse $\tau$ in $\bm_{ij}$. The two reverses of the tokens in $\bp$ and $\widetilde\bp$, form a pair of mutually reverse tokens $\{\tau,\tilde\tau\}$ in~$\bm$. Such a pair can be obtained for any token $\tau$ in $\bp$. Augmenting the set of all those pairs by the set of mutually reverse tokens in $\bm_{0i}$, $\bm_{ij}$ and $\bm_{j0}$, we obtain a partition of the set $\CCC(\bm)$ into pairs of mutually reverse tokens, which establishes that  the message $\bm$ is vacuous. 
\vspace{.1cm}

We have shown that the token system $(\SSS,\TTT)$ satisfies Axioms [Ma] and [Mb].  The proof is thus complete. \EOP

\begin{remark}\label{not simple}{\rm In the above proof, the inductive argument used to establish Case 2 of [M3] may convey the mistaken impression that the situation is always straightforward. The simple graph pictured in Figure \ref{AxiomM3} is actually glossing over some intricacies. The non--isometric subgraph $\BW$ is pictured by the thick lines in Figure \ref{multi-equiv} and is not `convex.' We can see how the inductive stage splitting the closed walk $\BW$  by the shortest path $\BL$ may lead to form, in each of the two smaller closed walks, pairs
$\{\mu,\tilde\mu\}$ and $\{\nu,\tilde\nu\}$ which correspond in fact to the same pair of tokens in $\BW$. Since the arcs corresponding to $\mu$ and $\nu$ are in the like relation $\bmfL$, the mistaken assignment is temporary.

{\begin{figure}[h!]
\centerline{\includegraphics*[scale=0.27]{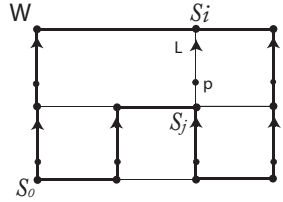}}
\caption[{multi-equiv}]{\label{multi-equiv} {\sc The non--isometric subgraph $\BW$ of Case 2 in the proof of [M3] in Theorem \ref{pair induced by G} is pictured in thick lines. The inductive stage of the proof
 leads to form temporarily, in each of the two smaller closed walks delimited by the shortest path $\BL$, pairs
$\{\mu,\tilde\mu\}$ and $\{\nu,\tilde\nu\}$ corresponding to the same pair of mutually reverse tokens in $\BW$. 
}}
\end{figure}
}
}
\end{remark}

We finally obtain:

\begin{theorem}\label{1-1 cor M-G}{\sl Let $\SSS$ an arbitrary set, with $|\SSS|\geq 2$. Denote by $\mfM$ the set of all media on $\SSS$, and by $\mfG$ the set of all mediatic graphs on~$\SSS$. There exists a bijection $\mff:\mfM\to\mfG\,:\,\MMM\mapsto\mff(\MMM)$ such that $G=\mff(\MMM)$ is the graph of $\MMM$ in the sense of Definition {\rm \ref{so1 media graph}} if and only if $\MMM$ is the medium of the mediatic graph $G$ in the sense of Definition {\rm\ref{graph -> token}}.}
\end{theorem} 

\begin{proof} Because the set~$\SSS$ of states is constant in $\mfM$ and confounded with the constant set of vertices in $\mfG$, we could reinterpret the function $\mff$ as a mapping of the family $\mfT$ of all sets of token $\TTT$ making $(\SSS,\TTT)$ a medium, into the family $\mfE$ of all sets of edges $E$ making $(\SSS,E)$ a mediatic graph. However, any set of edges $E$ of a mediatic graph on $\SSS$ is characterized by its like relation $\bmfL$, or equivalently, by the partition  of $\vec E$ induced by $\bmfL$. We choose the latter characterization for the purpose of this proof, and denote by $\vec \mfE_{| lr}$ the set of all the partitions of the sets of arcs $\vec E$ induced by the like relations characterizing the sets of edges in the collection~$\mfE$. 
\tl
 From Lemmas  \ref{media graph => mediatic} and \ref{pair induced by G}, we know that the graph of a medium is mediatic, and that the token system induced by a mediatic graph is a medium.  We have to show that the functions
 $$ \mff: \mfT\to \vec \mfE_{| lr}\quad\text{and}\quad\mfg: \vec \mfE_{| lr}\to\mfT
 $$ implicitly defined by (\ref{def graph rep}) and   (\ref{def tau}), respectively, are mutual inverses.
Note that, for any $\TTT\in\mfT$, the partition $\mff(\TTT)$ is defined via a function $f$ mapping $\TTT$ into the partition $\mff(\TTT)$. Writing as before $\langle ST\rangle$ for the equivalence class containing the arc $ST$, we have 
\begin{equation}\label{def mff f}
P\tau = Q\quad\EQ\quad f(\tau)= \langle PQ\rangle\quad\qquad(\tau\in\TTT;\, P,Q\in\SSS).
\end{equation} 
Proceeding similarly, but inversely, for the function $\mfg$, we notice that it defines, for each $\vec E_{| lr}$ in $\vec \mfE_{| lr}$ the set of tokens $\mfg(\vec E_{| lr})$ via a function $g$ mapping $\vec E_{| lr}$ into the set of tokens $\mfg(\vec E_{| lr})$; we obtain
\begin{equation}\label{def mfg g}
\langle ST\rangle = \langle PQ\rangle\quad\EQ\quad Pg(\langle ST\rangle) =Q\qquad (S,T,P,Q\in\SSS).
\end{equation}
Combining (\ref{def mff f}) and (\ref{def mfg g}) we obtain
$$
P\tau = Q\EQ f(\tau)= \langle PQ\rangle\EQ P(g\circ f)(\tau) =Q\quad (\tau\in\TTT;\, P,Q\in\SSS).
$$
We have thus $g = f^{-1}$ and so $\mfg = \mff^{-1}$. Conversely, we have 
\begin{align*}
\langle ST\rangle = \langle PQ\rangle\EQ  Pg(\langle ST\rangle) =Q\EQ (f\circ g)(\langle ST\rangle) 
&= \langle PQ\rangle\\
&(S,T,P,Q\in\SSS),
\end{align*}
yielding $f = g^{-1}$ and so $\mff = \mfg^{-1}$.
\end{proof}
\newtheorem{twoexamples}[theorem]{Two Examples}
\begin{twoexamples} {\rm In the last paragraph of our introductory section, we announced that the collection $\mfI$ of all the interval orders on a finite set $X$ was representable as a mediatic graph. The argument goes as follows. Doignon and Falmagne (1997) proved that such a collection $\mfI$ is always `well-graded', that is, for any two interval orders $K$ and $L$, there exists a sequence $K_0 = K, K_1,\ldots, K_n = L$ of interval orders on $X$ such that $|K_i\bigtriangleup K_{i+1}| = 1$ for $0 \leq i \leq n-1$ and $|K\bigtriangleup L| = n$. It is easily shown  \citep[see][]{falma97a} that any well-graded family $\FFF$ can be cast as a medium $\MMM(\FFF)$: the states of the medium are the sets of the family, and the tokens consist in either adding or removing an element from a set in $\FFF$. By Theorem \ref{media graph => mediatic}, the graph of the medium $\MMM(\III)$ is mediatic. A similar argument applies to the family of all the semiorders on~$X$, and to some other families on~$X$ (for example, partial orders and biorders, cf.~Doignon and Falmagne, 1997).
}
\end{twoexamples} 

\subsubsection*{References}
\begin{itemize}
\bibitem[Bondy(1995)]{bondy95}
\hspace*{-.7cm}J.A. Bondy.
\newblock {Basic graph theory: paths and circuits}.
\newblock In R.L. Graham, M.~Gr{\"o}tschel, and L.~Lov{\'a}sz, editors,
  \emph{Handbook of Combinatorics}, volume~1. The M.I.T. Press, Cambridge, MA,
  1995.

\bibitem[Djokovi{\'c}(1973)]{djoko73}
\hspace*{-.7cm}D.Z. Djokovi{\'c}.
\newblock {Distance preserving subgraphs of hypercubes}.
\newblock \emph{Journal of Combinatorial Theory, Ser. B}, 14:\penalty0
  263--267, 1973.
  
\bibitem[Doignon and Falmagne(1997)]{doign97}
\hspace*{-.7cm}J.-P. Doignon, and J.-Cl.~Falmagne.
\newblock {Well-graded families of relations}.
\newblock \emph{Discrete Mathematics}, 173:\penalty0 35--44, 1997.
  
  \bibitem[Eppstein(2002)]{eppst02}
\hspace*{-.7cm}D.~Eppstein and Falmagne, J.-Cl.
\newblock {Algorithms for media}.
\newblock \emph{Electronic preprint}, arXiv.org, cs.DS/0206033),
  2002.

\bibitem[Eppstein(2005)]{eppst05}
\hspace*{-.7cm}D.~Eppstein.
\newblock {The lattice dimension of a graph}.
\newblock \emph{European Journal of Combinatorics}, 26(6):\penalty0 585--592,
  2005.

\bibitem[Falmagne(1997)]{falma97a}
\hspace*{-.7cm}J.-Cl.~Falmagne.
\newblock {Stochastic token theory}.
\newblock \emph{Journal of Mathematical Psychology}, 41\penalty0 (2):\penalty0
  129--143, 1997.

\bibitem[Falmagne and Ovchinnikov(2002)]{falma02}
\hspace*{-.7cm}J.-Cl.~Falmagne and S.~Ovchinnikov.
\newblock {Media theory}.
\newblock \emph{Discrete Applied Mathematics}, 121:\penalty0 83--101, 2002.

\bibitem[Fishburn(1971)]{fishb71}
\hspace*{-.7cm}P.C. Fishburn.
\newblock {Betweenness, orders and interval graphs}.
\newblock \emph{Journal of Pure and Applied Algebra}, 1\penalty0 (2):\penalty0
  159--178, 1971.

\bibitem[Fishburn(1985)]{fishb85}
\hspace*{-.7cm}P.C. Fishburn.
\newblock \emph{{Interval orders and interval graphs}}.
\newblock John Wiley {\&} Sons, London and New York, 1985.

\bibitem[Fishburn and Trotter(1999)]{fishb99}
\hspace*{-.7cm}P.C. Fishburn and W.T. Trotter.
\newblock {Split semiorders}.
\newblock \emph{Discrete Mathematics}, 195:\penalty0 111--126, 1999.

\bibitem[Graham and Pollak(1971)]{graha71}
\hspace*{-.7cm}R.L. Graham and H.~Pollak.
\newblock {On addressing problem for loop switching}.
\newblock \emph{Bell Systems Technical Journal}, 50:\penalty0 2495--2519, 1971.

\bibitem[K{\"o}nig(1916)]{konig16}
\hspace*{-.7cm}D.~K{\"o}nig.
\newblock {{\"U}ber Graphen und ihren Anwendung auf Determinanten-theorie und
  Mengenlehre}.
\newblock \emph{Mathematische Annalen}, 77:\penalty0 453--465, 1916.

\bibitem[Ovchinnikov(2006{\natexlab{a}})]{ovchi06a}
\hspace*{-.7cm}S.~Ovchinnikov.
\newblock {Media theory: representations and examples}.
\newblock \emph{Discrete Applied Mathematics}, 2006.
\newblock Accepted for publication.

\bibitem[Ovchinnikov and Dukhovny(2000)]{ovchi00}
\hspace*{-.7cm}S.~Ovchinnikov and A.~Dukhovny.
\newblock {Advances in media theory}.
\newblock \emph{International Journal of Uncertainty, Fuzziness and
  Knowledge-Based Systems}, 8\penalty0 (1):\penalty0 45--71, 2000.
  
  \bibitem[Winkler(1984)]{winkl84}
\hspace*{-.7cm}P.M. Winkler.
\newblock {Isometric embedding in products of complete graphs}.
\newblock \emph{Discrete Applied Mathematics}, 7:\penalty0 221--225, 1984.
\end{itemize}
\end{document}